\documentclass[12pt]{article}
\usepackage[a4paper, margin=2.3cm]{geometry}
\usepackage{amsmath,amssymb,amsthm}
\usepackage{color}
\usepackage{parskip}

\newtheorem{theorem}{Theorem}[section]
\newtheorem{remark}{Remark}[section]

\newtheorem{corollary}[theorem]{Corollary}
\newtheorem{lemma}[theorem]{Lemma}
\newtheorem*{definition*}{Definition}

\newcommand\be{\begin{eqnarray*}}
\newcommand\ee{\end{eqnarray*}}
\newcommand\beq{\begin{equation}}
\newcommand\eeq{\end{equation}}

\newcommand\ben{\begin{eqnarray}}
\newcommand\een{\end{eqnarray}}

\def\E{\mathcal{E}}

\begin{document}
\title{A new bound on Erd\H{o}s distinct distances problem in the plane over prime fields
}
\author{A. Iosevich\thanks{Department of Mathematics, University of Rochester. Email: {\tt iosevich@math.rochester.edu}}\and
D. Koh\thanks{Department of Mathematics, Chungbuk National University. Email: {\tt koh131@chungbuk.ac.kr}}
\and 
    T. Pham\thanks{Department of Mathematics,  UCSD. Email: {\tt v9pham@ucsd.edu}}
  \and
  C-Y. Shen \thanks{Department of Mathematics,  National Taiwan University. Email: {\tt cyshen@math.ntu.edu.tw}}
  \and
  L. Vinh \thanks{Vietnam Institute of Educational Sciences . Email: {\tt vinhla@vnu.edu.vn}}}

\date{}
\maketitle
\date{}
\maketitle
\begin{abstract}
In this paper we obtain a new lower bound on the Erd\H{o}s distinct distances problem in the plane over prime fields. More precisely, we show that for any set $A\subset \mathbb{F}_p^2$ with $|A|\le p^{7/6}$ and $p\equiv 3\mod 4$, the number of distinct distances determined by pairs of points in $A$ satisfies 
$$ |\Delta(A)| \gtrsim |A|^{\frac{1}{2}+\frac{149}{4214}}.$$ 
Our result gives a new lower bound of $|\Delta{(A)}|$ in the range
$|A|\le p^{1+\frac{149}{4065}}$.

 The main tools in our method are the energy of a set on a paraboloid due to Rudnev and Shkredov, a point-line incidence bound given by Stevens and de Zeeuw, and a lower bound on the number of distinct distances between a line and a set in $\mathbb{F}_p^2$. The latter is the new feature that allows us to improve the previous bound due Stevens and de Zeeuw. 
\end{abstract}
\section{Introduction}

The celebrated Erd\H{o}s distinct distances problem asks for the minimum number of
distinct distances determined by a set of $n$ points in the plane over the real numbers. The breakthrough work of Guth and Katz \cite{guth} shows that a set of $n$ points in $\mathbb{R}^2$ determines at least $Cn/\log(n)$ distinct distances. The same problem can be considered in the setting of finite fields.

Let $\mathbb{F}_p$ be the prime field of order $p$. The ``distance" formula between two points $x=(x_1, x_2)$ and $y=(y_1, y_2)$ in $\mathbb{F}_p^2$ is defined by
\[||x-y||:=(x_1-y_1)^2+(x_2-y_2)^2.\]

While this is not a distance in the traditional sense, the definition above is a reasonable analog of the Euclidean distance in that it is invariant under orthogonal transformations. 

For $A\subset \mathbb{F}_p^2$, let 
$$ \Delta(A)=\{||x-y||: x,y \in A \}$$ and let $|\Delta(A)|$ denote its size. It has been shown in a remarkable paper of Bourgain, Katz, and Tao \cite{bkt} that if $|A|=p^\alpha$, $0<\alpha<2$, then we have 
$$ |\Delta(A)| \ge |A|^{\frac{1}{2}+\varepsilon},$$
for some $\varepsilon=\varepsilon(\alpha)>0$. 

This result has been quantified and improved over time. The recent work of Stevens and De Zeeuw \cite{ffrank} shows that 
\begin{equation}\label{S-Z} |\Delta(A)| \ge |A|^{\frac{1}{2}+\frac{1}{30}}=|A|^{\frac{8}{15}}, \end{equation}
under the condition $|A| \ll p^{\frac{15}{11}}$.  

Here and throughout, $X \ll Y$ means that there exists $c_1>0$, independent of $p$, such that $X \leq c_1Y$, $X \gtrsim Y$ means $X\gg (\log Y)^{-c_2} Y$ for some positive constant $c_2$, and $X\sim Y$ means that $c_3X\le Y\le c_4X$ for some positive constants $c_3$ and $c_4$.

For the case of large sets, Iosevich and Rudnev \cite{ir} used Fourier analytic methods to prove that for $A\subset \mathbb{F}_q^d$, where $q$ is not necessarily prime, with $|A|\ge 4q^{\frac{d+1}{2}}$, we have $\Delta(A)={\Bbb F}_q$. It was shown in \cite{hart} that the threshold $q^{\frac{d+1}{2}}$ cannot in general be improved when $d$ is odd, even if we wish to recover a positive proportion of all the distances in ${\Bbb F}_q$. In prime fields, the question is open in dimension $3$ and higher. In two dimensions, Chapman, Erdogan, Koh, Hart and Iosevich (\cite{CEHIK10}) proved that if $|A| \ge p^{\frac{4}{3}}$, $p$ prime, then $|\Delta(A)| \gg p$. In particular, their proof shows that if $ Cp\le |A|\le p^{4/3}$ for a sufficiently large $C>0,$ then 
\begin{equation}\label{CEKHIexponent} |\Delta(A)| \gg \frac{|A|^{3/2}}{p}.\end{equation}  The $4/3$ threshold  was extended to all (not necessarily prime) fields by Bennett, Hart, Iosevich, Pakianathan and Rudnev (\cite{BHIPR17}). We refer the reader to \cite{ir, hart} for further details.

The main purpose of this paper is to improve the exponent $\frac{1}{2}+\frac{1}{30}=\frac{8}{15}$ on the magnitude of $\Delta(A)$ when $A$ is a relatively small set in $\mathbb{F}_p^2$ with $p\equiv 3\mod 4$.  The main tools in our arguments are the energy of a set on a paraboloid due to Rudnev and Shkredov, a point-line incidence bound given by Stevens and de Zeeuw, and a lower bound on the number of distinct distances between a line and a set in $\mathbb{F}_p^2$. The following is our main result.
\bigskip
\begin{theorem}\label{thm1}
Let $\mathbb{F}_p$ be a prime field of order $p$ with $p\equiv 3\mod 4$. For $A\subset \mathbb{F}_p^2$ with $|A|\ll p^{\frac{7}{6}}$,  we have 
$$ |\Delta(A)| \gtrsim |A|^{\frac{1128}{2107}}=|A|^{\frac{1}{2} + \frac{149}{4214}}.$$
\end{theorem}

\begin{remark} The Stevens-de Zeeuw exponent in \eqref{S-Z} is $.533...$, whereas our exponent is $.535358...$. Thus our result is better than that of the Stevens-de Zeeuw in the range $|A|\ll p^{7/6}.$ On the other hand, our result is superior to \eqref{CEKHIexponent} in the range $|A|\le p^{\frac{1}{\frac{3}{2}-\frac{1128}{2107}}}=p^{\frac{4214}{4065}}.$ 
In conclusion, Theorem \ref{thm1} improves the currently known distance results in the range $|A|\ll p^{\frac{4214}{4065}}.$
 \end{remark} 

\vskip.125in 

\begin{remark} While our improvement over the Steven-de Zeeuw estimate is small, we introduce a new idea, namely the count for the number of distances between a line and a set. This should lead to further improvements in the exponent in the future. \end{remark} 

The rest of the paper is devoted to prove Theorem \ref{thm1}, and we always assume that $p\equiv 3\mod 4$. 

\paragraph{Acknowledgments:}The authors would like to thank  Frank de Zeeuw for useful discussions and comments. A. Iosevich was partially supported by the NSA Grant H98230-15-1-0319. D. Koh was supported by Basic Science Research Programs through National Research Foundation of Korea funded by the Ministry of Education (NRF-2018R1D1A1B07044469). T. Pham was supported by Swiss National Science Foundation
grant P400P2-183916. C-Y Shen was supported in part by MOST, through grant 104-2628-M-002-015 -MY4.

\section{Proof of Theorem \ref{thm1}}
To prove Theorem \ref{thm1} we make use of the following lemmas. The first lemma is a point-line incidence bound due to Stevens and De Zeeuw in \cite{ffrank}. 
\bigskip
\begin{lemma}[\cite{ffrank}]\label{ff}
Let $P$ be a set of $m$ points in $\mathbb{F}_p^2$ and $L$ be a set of $n$ lines in $\mathbb{F}_p^2$. Suppose that $m^{7/8}\le n\le m^{8/7}$ and $m^{-2}n^{13}\ll p^{15}.$ Then we have 
\[I(P, L)=\#\{ (p,\ell) ; p \in P, \ell \in L   \}\ll m^{11/15}n^{11/15}.\]
\end{lemma}
\bigskip
Let $P$ be a paraboloid in $\mathbb{F}_p^3$. For $Q\subset P$, let $E(Q)$ be the additive energy of the set $Q$, namely, the number of tuples $(a, b, c, d)\in Q^4$ such that $a-b=c-d.$ Using Pach and Sharir's argument in \cite{pach} and Lemma \ref{ff}, Rudnev and Shkredov \cite{R-S} derived an upper bound of $E(Q)$ as follows.
\bigskip
\begin{lemma}[\cite{R-S}]\label{rud}
Let $P$ be a paraboloid in $\mathbb{F}_p^3$. For $Q\subset P$ with $|Q|\ll p^{26/21}$, we have 
\[E(Q)\ll |Q|^{17/7}.\]
\end{lemma}
\bigskip
In the following theorem, we give a lower bound on the number of distinct distances between a set on a line and an arbitrary set in $\mathbb{F}_p^2$. This will be a crucial step in the proof of Theorem \ref{thm1}.  The precise statement is as follows. 
\bigskip
\begin{theorem}\label{thm-any}
Let $l$ be a line in $\mathbb{F}_p^2$, $P_1$ be a set of points on $l$, and $P_2$ be an arbitrary set in $\mathbb{F}_p^2$. Suppose that $|P_1|^{\frac{4}{7}}< |P_2|\ll p^{\frac{7}{6}}.$ Then the number of distinct distances between $P_1$ and $P_2$, denoted by $|\Delta(P_1, P_2)|$, satisfies
$$ |\Delta(P_1, P_2)| \gtrsim \min\left\lbrace |P_1|^{\frac{4}{11}} |P_2|^{\frac{4}{11}}, ~|P_1||P_2|^{1/8}, |P_2|^{7/8}, |P_1|^{-1}|P_2|^{8/7}\right\rbrace.$$
\end{theorem}

We will provide a detailed proof of Theorem \ref{thm-any} in Section $3$.  The following is a direct consequence from Theorem \ref{thm-any}. 
\bigskip

\begin{corollary}\label{co99}
Let $A\subset \mathbb{F}_p^2$ with $|A|\ll p^{7/6}$. Suppose there is a line containing at least $|A|^{\frac{7}{15}+\epsilon}$ points from $A$. Then we have 
$$ |\Delta(A)| \gtrsim \min \{|A|^{\frac{8}{15}+\frac{4\epsilon}{11}}, |A|^{\frac{8}{15}+\frac{1}{7}-\epsilon}\}.$$
\end{corollary}
The above corollary shows that the exponent $8/15$ in \eqref{S-Z} due to  Stevens and De Zeeuw is improved when $A$ contains many points on a line.

\bigskip
We are now ready to prove Theorem \ref{thm1}.
\bigskip
\paragraph{Proof of Theorem \ref{thm1}:}
Let $\epsilon>0$ be a parameter chosen at the end of the proof. Throughout the proof, we assume that that 
\begin{equation}\label{assep}\frac{16}{15}+2\epsilon<\frac{8}{7},\end{equation}
which is equivalent with $\epsilon <4/105$.
If there is a line containing at least $|A|^{7/15+\epsilon}$ points from $A$, then we obtain  by Corollary \ref{co99} that 
\begin{equation}\label{bigpl}  
|\Delta(A)| \gg \min \{|A|^{\frac{8}{15}+\frac{4\epsilon}{11}}, |A|^{\frac{8}{15}+\frac{1}{7}-\epsilon}\}.
\end{equation}

  Now we assume that there is no line supporting more than $|A|^{7/15+\epsilon}$ points from $A$. 

For any line $l$ in $\mathbb{F}_p^2$ defined by the equation $ax+by-c=0$, the vector $(a, b, c)$ is called a vector of parameters of $l$.

We first start with counting the number of triples $(z, x, y)\in A^3$ such that $||z-x||=||z-y||$, where $z=(a, b)$, $x=(x_1, x_2), y=(y_1, y_2)$. 

It follows from the equation $||z-x||=||z-y||$ that
\[(-2a)(x_1-y_1)+(-2b)(x_2-y_2)+(x_1^2+x_2^2)-(y_1^2+y_2^2)=0.\]
This equation defines a line in $\mathbb{F}_p^2$ with the parameters
\[(x_1, x_2, x_1^2+x_2^2)-(y_1, y_2, y_1^2+y_2^2)=(x_1-y_1, x_2-y_2, x_1^2+x_2^2-y_1^2-y_2^2).\]
Let $L$ be the set of these lines. It is clear that $L$ can be a multi-set. 

Let $Q$ be the set of points of the form $(x, y, x^2+y^2)$ with $(x, y)\in A$. We have $Q$ is a set on the paraboloid $z=x^2+y^2$ and $|Q|=|A|$.
 
Notice that the number of triples $(z, x, y) \in A^3$ with the property $||z-x||=||z-y||$ is equivalent to the number of incidences between lines in $L$ and points in $-2A:=\{(-2a_1, -2a_2): (a_1, a_2)\in A\}$. 

For each line $l$ in $L$, let $f(l)$ be the size of $l\cap (-2 A)$, and $m(l)$ be the multiplicity of $l$. Let $L_1$ be the set of distinct lines in $L$.

Thus, we have 
\begin{align*}
I(-2A, L)&=\sum_{l\in L_1}f(l)m(l)\\&
=\sum_{l\in L_1, f(l)\le |A|^{7/15-\epsilon}}f(l)m(l)+\sum_{l\in L_1,  |A|^{7/15-\epsilon}\le f(l)\le |A|^{7/15+\epsilon}}f(l)m(l)\\
&=I_1+I_2.
\end{align*}
We now bound $I_1$ and $I_2$ as follows. 

One can check that the size of $L$ is bounded by $|A|^2$, which implies that 
\[I_1\le |A|^{\frac{37}{15}-\epsilon}.\]

Let $L_2$ be the set of distinct lines $l$ in $L_1$ such that $|A|^{\frac{7}{15}-\epsilon}\le f(l)\le |A|^{\frac{7}{15}+\epsilon}$.

To bound $I_2$, we consider the following two cases:

{\bf Case $1$:} Suppose
\[\sum_{l\in L_2}m(l)\le |A|^{2-\frac{15\epsilon}{11}}.\]
We see that 
\[ I_2=\sum_{l\in L_2}f(l)m(l)\le |A|^{\frac{37}{15}-\frac{4\epsilon}{11}},\]
since any line in $L_2$ contains at most $|A|^{7/15+\epsilon}$ points.
Thus in this case we obtain that
\begin{equation}\label{MAE} I(-2A, L) =I_1+I_2 \le |A|^{\frac{37}{15}-\epsilon} + |A|^{\frac{37}{15}-\frac{4\epsilon}{11}} \ll |A|^{\frac{37}{15}-\frac{4\epsilon}{11}}.\end{equation}  
Now, for each $t\in \mathbb F_p,$ let $\nu(t)$ denote the number of pairs $(x,y)\in A^2$ such that $\|x-y\|=t.$ We have
$$\nu^2(t)=\left(\sum_{x,y\in A: \|x-y\|=t} 1\right)^2 = \left(\sum_{x\in A} 1\times (\sum_{y\in A: \|x-y\|=t}1)\right)^2.$$ 
By the Cauchy-Schwarz inequality, 
$$\nu^2(t)\le |A| \sum_{x\in A} \left(\sum_{y\in A: \|x-y\|=t}1\right)^2 = |A| \sum_{x,y,z\in A: \|x-y\|=t=\|x-z\|} 1.$$
Summing over $t\in \mathbb F_p,$ we obtain
$$ \sum_{t\in \mathbb F_p} \nu^2(t)\le |A| \sum_{x,y,z\in A: ||x-y\|=\|x-z\|} 1.$$
By the Cauchy-Schwarz inequality and the above inequality, we get
$$ \frac{|A|^4}{|\Delta(A)|} \le \sum_{t\in \mathbb F_p} \nu^2(t) \le |A|  \#\{(x,y,z) \in A^3 ; ||x-y||=||x-z||\}\ll |A| I(-2A, L).$$
Combining the above inequality with \eqref{MAE},
we obtain
\begin{equation}\label{verynice1} |\Delta(A)|\gg |A|^{\frac{8}{15}+\frac{4\epsilon}{11}}.\end{equation}

{\bf Case $2$:} Suppose
\[\sum_{l\in L_2}m(l)\ge |A|^{2-\frac{15\epsilon}{11}}.\]
By the Cauchy-Schwarz inequality and Theorem \ref{rud}, we have 
\begin{align}\label{eqx90}
&\# \left\lbrace(a-b, ||a||-||b||)\colon a, b\in A, (a-b, ||a||-||b||) ~\mbox{is a vector of parameters of a line in}~ L_2  \right\rbrace\\&\gg \frac{\left(\sum_{l\in L_2}m(l)\right)^2}{E(Q)}\nonumber
\\&\gg |A|^{\frac{11}{7}-\frac{30\epsilon}{11}}\nonumber.\end{align}

In the next step, we are going to show that 
\[|L_2|\le |A|^{1+\frac{15\epsilon}{4}}.\]
Indeed, since each line in $L_2$ contains at least $|A|^{7/15-\epsilon}$ points, the size of $L_2$ is at most $|A|^{16/15+2\epsilon}\ll |A|^{8/7}$. The last inequality follows from our assumption \eqref{assep}. Hence, we are able to apply Theorem \ref{ff} so that we have 
\[|A|^{\frac{7}{15}-\epsilon}|L_2|\le I(-2A, L_2)\le |A|^{11/15}|L_2|^{11/15},\]
which gives us that
\begin{equation}\label{eqx91}|L_2|\ll |A|^{1+\frac{15\epsilon}{4}}.\end{equation}

For each line $l\in L_2$, let $m'(l)$ be the number of distinct vectors $(a-b, ||a||-||b||)$ with $(a, b)\in A^2$ such that $(a-b, ||a||-||b||)$ is a vector of parameters of $l$. 

It follows from (\ref{eqx90}) and (\ref{eqx91}) that 
there exists $l\in L_2$ such that 
\begin{equation}\label{eqAdidaphat}m'(l)\gg |A|^{\frac{4}{7}-\frac{30\epsilon}{11}-\frac{15\epsilon}{4}}.\end{equation}

We now claim that $|\Delta(A)|\gg m'(l)$. Indeed, suppose that $l$ is determined by $m'(l)$ distinct vectors $(a_1-b_1, ||a_1||-||b_1||), \ldots, (a_{m'(l)}-b_{m'(l)}, ||a_{m'(l)}||-||b_{m'(l)}||)$. Then we have 
\begin{align*}
&(a_2-b_2, ||a_2||-||b_2||)=\lambda_2 \cdot (a_1-b_1, ||a_1||-||b_1||), \\
&(a_3-b_3, ||a_3||-||b_3||)=\lambda_3 \cdot (a_1-b_1, ||a_1||-||b_1||),\\
&...............................................................................\\
&(a_{m'(l)}-b_{m'(l)}, ||a_{m'(l)}||-||b_{m'(l)}||)=\lambda_{m'(l)}\cdot (a_1-b_1, ||a_1||-||b_1||),
\end{align*}
for some $\lambda_2, \ldots, \lambda_{m'(l)}\in \mathbb{F}_p$. Since the vectors $(a_1-b_1, ||a_1||-||b_1||), \ldots,$ and  $(a_{m'(l)}-b_{m'(l)}, ||a_{m'(l)}||-||b_{m'(l)}||)$ are distinct, we have $\lambda_2, \ldots, \lambda_{m'(l)}$ are distinct. On the other hand, we also have 
\[||a_2-b_2||=\lambda_2^2 \cdot ||a_1-b_1||, \ldots, ||a_{m'(l)}-b_{m'(l)}||=\lambda_{m'(l)}^2\cdot ||a_1-b_1||,\]
which gives us $|\Delta(A)|\ge \frac{m'(l)-1}{2}$, and the claim is proved. 

Hence, it follows from the equation (\ref{eqAdidaphat}) that 

\begin{equation}\label{case2R} |\Delta(A)| \gg |A|^{\frac{4}{7}-\frac{30\epsilon}{11}-\frac{15\epsilon}{4}}.\end{equation}

By \eqref{verynice1} of \textbf{Case 1} and \eqref{case2R} of \textbf{Case 2}, it follows that if no line contains more than $|A|^{\frac{7}{15}+\epsilon}$ points in $A$, then
$$ |\Delta(A)|\gg \min \left\{|A|^{\frac{8}{15}+\frac{4\epsilon}{11}},~|A|^{\frac{4}{7}-\frac{30\epsilon}{11}-\frac{15\epsilon}{4}} \right\}.$$

Finally, combining this fact with \eqref{bigpl} yields that
$$ |\Delta(A)|\gg \min\left\{ |A|^{\frac{8}{15}+\frac{4\epsilon}{11}},~|A|^{\frac{8}{15}+\frac{1}{7}-\epsilon},~|A|^{\frac{4}{7}-\frac{30\epsilon}{11}-\frac{15\epsilon}{4}} \right\}.$$

To deduce the desirable result, we consier the common solutions $(\epsilon, \delta)$ to the system of the following three inequalities:
$$ \frac{8}{15}+\frac{4\epsilon}{11} \ge \delta, ~~ \frac{8}{15}+\frac{1}{7}-\epsilon\ge \delta,~~\frac{4}{7}-\frac{30\epsilon}{11}-\frac{15\epsilon}{4}\ge \delta.$$
By a direct computation, we can obtain the largest $\delta=\frac{1128}{2107}$  for $\epsilon=\frac{176}{31605}.$ 
Thus, choosing $\epsilon=\frac{176}{31605}$ gives 
$$ |\Delta(A)|\gg |A|^\delta=|A|^{\frac{1128}{2107}},$$
which completes the proof. $\hfill\square$

\section{Distances between a set on a line and an arbitrary set in $\mathbb{F}_p^2$}
In this section, we will prove Theorem \ref{thm-any}. We first start with an observation as follows: if \[ |\Delta(P_1, P_2)| \gg \min \left\lbrace |P_2|^{8/7}|P_1|^{-1}, |P_2|^{7/8}\right\rbrace,\]
then we are done. So WLOG, we assume that 
\begin{equation}\label{eq919Adidaphat} |\Delta(P_1, P_2)| \ll \min \left\lbrace |P_2|^{8/7}|P_1|^{-1}, |P_2|^{7/8}\right\rbrace.\end{equation}
Hence, to prove Theorem \ref{thm-any}, it is sufficient to show that
\[ |\Delta(P_1, P_2)| \gtrsim
\min\left\lbrace |P_1|^{\frac{4}{11}} |P_2|^{\frac{4}{11}}, ~|P_1||P_2|^{1/8}\right\rbrace.\]
Since the distance function is preserved under translations and rotations, we can assume that the line is vertical passing through the origin, i.e. $P_1\subset \{0\}\times \mathbb{F}_p$. For the simplicity, we identify each point in $P_1$ with its second coordinate. The following lemma on a point-line incidence bound is known as a direct application of the K\H ovari--S\'os--Tur\'an theorem in \cite{bo}.
\bigskip
\begin{lemma}\label{cauchy-incidence}
Let $P$ be a set of $m$ points in $\mathbb{F}_p^2$ and $L$ be a set of $n$ lines in $\mathbb{F}_p^2$. We have 
\[I(P, L)\le \min\left\lbrace m^{1/2}n+m, n^{1/2}m+n\right\rbrace.\]
\end{lemma}
For $x\in P_1$ and $P_2\subset \mathbb{F}_p^2$, we define
\[\mathcal{E}(P_2, x):=\#\left\lbrace \left((a, b), (c, d)\right)\in P^2_{2}\colon a^2+(b-x)^2=c^2+(d-x)^2\right\rbrace,\]
as the number of pairs of points in $P_2$ with the same distance to $x\in P_1$. In the next lemma, we will give an upper bound for $\sum_{x\in P_1}\mathcal{E}(P_2, x)$. 
\bigskip
\begin{lemma}\label{maillemma}
Let $P_1, P_2$ be sets as in Theorem \ref{thm-any}. Suppose that $|P_1|^{4/7}< |P_2|$ and $|P_2|\ll p^{{7}/{6}}$. Then we have
\[\sum_{x\in P_1}\E(P_2, x)\lesssim |P_1|^{7/11}|P_2|^{18/11}+|P_2|^{15/8}.\]
\end{lemma}
\begin{proof}
For $x\in P_1$ and $\lambda\in \mathbb{F}_p$, let $r_{P_2}(x, \lambda)$ be the number of points $(a, b)$ in $P_2$ such that $a^2+(b-x)^2=\lambda$. Then we have 
\[T:=\sum_{x\in P_1}\E(P_2, x)=\sum_{(x, \lambda)\in P_1\times \mathbb{F}_p}r_{P_2}(x, \lambda)^2.\]
Let $t=\frac{|P_2|^{7/11}}{|P_1|^{4/11}}> 1$, and let $R_t$ be the number of pairs $(x, \lambda)\in P_1\times \mathbb{F}_p$ such that $r_{P_2}(x, \lambda)\ge t$. We have 
\[T=\sum_{(x, \lambda)\not \in R_t}r_{P_2}(x, \lambda)^2+\sum_{(x, \lambda)\in R_t}r_{P_2}(x, \lambda)^2=I+II.\]
Since $\sum_{(x, \lambda)\not\in R_t}r_{P_2}(x, \lambda)\le |P_1||P_2|$ and $r_{P_2}(x, \lambda)<t$ for any pair $(x, \lambda)\not\in R_t$, we have \[I\le t|P_1||P_2|=|P_1|^{7/11}|P_2|^{18/11}.\]

In the next step, we will bound $II$. 

From the equation $\lambda=a^2+(b-x)^2$, we have 
\[a^2+b^2=2bx-x^2+\lambda.\]
Let $P$ be the set of points $(b, a^2+b^2)$ with $(a, b)\in P_2$, and $L$ be the set of lines defined by $y=2ux-u^2+v$ with $(u, v)\in R_t$. We have $|L|=|R_t|$ and $|P|\sim |P_2|$. 

With these definitions,  we observe that $II$ can be viewed as the number of pairs of points in $P$ on lines in $L$. 

We partition $L$ into at most $\log (|P|)$ sets of lines $L_i$ as follows:
\[L_i=\{l\in L\colon 2^it\le |l\cap P|<2^{i+1}t\},\]
and let $II(L_i)$ denote the number of pairs of points in $P$ on lines in $L_i$. 

For each $i$, we now consider the following cases:

{\bf Case $1$:} $|P|^{1/2}<|L_i|\le |P|^{7/8}$.  It follows from Lemma \ref{cauchy-incidence} that 
\[2^it|L_i|\le I(P, L_i)\le |P|^{1/2}|L_i|+|P|\ll |P|^{1/2}|L_i|,\]
which leads to that $2^it\le |P|^{1/2}$. Thus 
\[II(L_i)\ll  |L_i| \left(|P|^{1/2}\right)^2\ll |P|^{15/8}\sim |P_2|^{15/8}.\]

{\bf Case $2$:}  $|P|^{7/8}\le |L_i| \le |P|^{8/7}$.  It follows from Lemma \ref{ff} that
\[2^it|L_i|\le I(P, L_i)\le |L_i|^{11/15}|P|^{11/15}.\]
This implies that 
\[|L_i|\le \frac{|P|^{11/4}}{(2^it)^{15/4}}.\]
In this case, we have
\[II(L_i)\le \frac{|P|^{11/4}}{(2^it)^{15/4}}\cdot 2^{2i+2}t^2\ll \frac{|P|^{11/4}}{(2^it)^{7/4}}\sim \frac{|P_2|^{11/4}}{(2^it)^{7/4}}.\]
One can check that the condition $m^{-2}n^{13}\ll p^{15}$ in the Theorem \ref{ff} is satisfied once $|P|\le p^{{7}/{6}}$.

{\bf Case $3$:} $|L_i|\le |P|^{1/2}$. Applying Lemma \ref{cauchy-incidence} again, we obtain
\begin{equation}\label{eqxAdidaphat}2^it|L_i|\le I(P, L_i)\le  |P|^{1/2}|L_i|+|P|\ll |P|.\end{equation}

If $2^it\ge |P|^{7/8}$, then there is at least one line in $L$ which has at least $|P|^{7/8}$ points from $P$, which follows that there exists $(x, \lambda)\in R_t$ such that the circle centered at $(0, x)$ of radius $\lambda$ contains at least $|P|^{7/8}\sim |P_2|^{7/8}$ points from $P_2$. This implies that 
\[|\Delta(P_1, P_2)|\gg |P_2|^{7/8},\]
which contradicts to our assumption (\ref{eq919Adidaphat}).

Thus, we can assume that $2^it\ll |P|^{7/8}$. With this condition, we have 
\[II(L_i)\ll |L_i|(2^it)^2\ll 2^it \cdot (|L_i|(2^it))\ll |P|^{15/8}\sim |P_2|^{15/8},\]
where we have used the inequality (\ref{eqxAdidaphat}) in the last step.

{\bf Case $4$:} $|L_i|\ge |P|^{8/7}$. In this case, by the pigeon-hole principle, there is a point $x$ in $P_1$ that determines at least $|P_2|^{8/7}/|P_1|$ lines, and each of these lines contains at least one point from $P$. This implies that 
\[ |\Delta(P_1, P_2)| \gg |P_2|^{8/7}|P_1|^{-1},\]
which contradicts to our assumption (\ref{eq919Adidaphat}).

Putting these cases together, and taking the sum over all $i$, we obtain 
\[T\lesssim |P_1|^{7/11}|P_2|^{18/11}+|P_2|^{15/8}.\]
This completes the proof of the lemma.
\end{proof}
We are ready to prove Theorem \ref{thm-any}.
\paragraph{Proof of Theorem \ref{thm-any}:}
As in the beginning of this section,  if \[ |\Delta(P_1, P_2)| \gg \min \left\lbrace |P_2|^{8/7}|P_1|^{-1}, |P_2|^{7/8}\right\rbrace,\]
then we are done. Thus, we might assume that 
\[ |\Delta(P_1, P_2)| \ll \min \left\lbrace |P_2|^{8/7}|P_1|^{-1}, |P_2|^{7/8}\right\rbrace.\]
Let $N$ be the number of quadruples $(p_1, p_2, p_1', p_2')\in P_1\times P_2\times P_1\times P_2$ such that 
\[||p_1-p_2||=||p_1'-p_2'||.\]

Let $T$ be the number of triples $(p_1, p_2, p_2')\in P_1\times P_2\times P_2$ such that 
$||p_1-p_2||=||p_1-p_2'||$. As in the proof of Lemma \ref{maillemma}, we have 
\[T\lesssim |P_1|^{7/11}|P_2|^{18/11}+|P_2|^{15/8}.\]
By the Cauchy-Schwarz inequality, we have 
\[N\ll |P_1|T\lesssim |P_1|^{18/11}|P_2|^{18/11}+|P_1||P_2|^{15/8}.\]
By the Cauchy-Schwarz inequality again, one can show that $\frac{|P_1|^2|P_2|^2}{|\Delta(P_1, P_2)|} \leq N$. Thus we have 
\[ |\Delta(P_1, P_2)| \gtrsim \min \{|P_1|^{4/11}|P_2|^{4/11}, ~|P_1||P_2|^{1/8}\}.\]
This ends the proof of the theorem. $\hfill\square$

\end{document}